\documentclass[12pt]{amsart}%
\usepackage{color}
\usepackage{amsmath}
\usepackage{amssymb}
\usepackage{amsfonts}
\usepackage[latin1]{inputenc}
\usepackage{graphicx}%
\usepackage{cite}
\providecommand{\U}[1]{\protect\rule{.1in}{.1in}}

\DeclareMathSymbol{\subsetneqq}{\mathbin}{AMSb}{36}

\theoremstyle{plain}
\numberwithin{equation}{section}
\newtheorem{theorem}{Theorem}[section]

\newtheorem{lemma}{Lemma}[section]

\newtheorem{definition}{Definition}[section]
\newtheorem{remark}{Remark}[section]

\begin{document}
\title[On The Strong Convergence of The Gradient Projection Algorithm ]
{On The Strong Convergence of The Gradient Projection Algorithm with Tikhonov regularizing term}%
\author{Ramzi May}%
\address{Mathematics and Statistics Department, College of Science, King Faisal University, P.O. 380 Ahsaa 31982, Kingdom of Saudi Arabia}
\email{rmay@kfu.edu.sa}
\vskip 0.2cm
\date{2 October, 2019}

\begin{abstract} We investigate the strong and the weak convergence
properties of the following  gradient projection algorithm with Tikhonov
regularizing term
\[
x_{n+1}=P_{Q}(x_{n}-\gamma _{n}\nabla f(x_{n})-\gamma _{n}\alpha _{n}\nabla
\phi (x_{n})),
\]
where $P_{Q}$ is the projection operator from a Hilbert space $\mathcal{H}$
onto a given nonempty, closed and convex subset $Q,$ $f:\mathcal{H}%
\rightarrow \mathbb{R}$ a regular convex function, $\phi :\mathcal{H}%
\rightarrow \mathbb{R}$ a regular strongly convex function, and $\gamma _{n}$
and $\alpha _{n}$ are positive real numbers. Following a Lyuapunov approach
inspired essentially from the paper [Comminetti R,
Peypouquet J Sorin S. Strong asymptotic convergence of evolution equations
governed by maximal monotone operators with Tikhonov regularization. J.
Differential Equations. (2001); 245:3753-3763], we establish the strong
convergence of $(x_{n})_{n}$ to a particular minimizer $x^{\ast }$ of $f$ on
$Q$ under some simple and natural conditions on the objective function $f$\
and the sequences $(\gamma _{n})_{n}$ and $(\alpha _{n})_{n}$.
\end{abstract}
\maketitle

\section{Introduction and main result}

Throughout this paper, $\mathcal{H}$\ is a given real Hilbert space endowed
with the inner product $\langle .,.\rangle $ and the associated norm $%
\left\Vert .\right\Vert ,$ $Q$ is a nonempty, closed and convex subset of $%
\mathcal{H}$ and $f:\mathcal{H}\rightarrow \mathbb{R}$ is a $C^{1}$\ convex
such that its gradient $\nabla f$ is $L_{f,Q}$-Lipschitz continuous on $Q$, i.e. there exits a constant $L_{f,Q}>0$ such that
\begin{equation}
\left\Vert \nabla f(x)-\nabla f(y)\right\Vert \leq L_{f,Q}\left\Vert
x-y\right\Vert ,~\forall x,y\in Q.  \label{C1}
\end{equation}%
We consider the following constrained convex minimization problem:
\begin{equation}
\text{min}\{f(x) : x\in Q\}.%
\tag{P}
\end{equation}%
We assume that (P) has at least one solution and we denote by $S_{f,Q}$ the
set of its solutions:%
\[
S_{f,Q}=\{x\in Q:f(x)=f_{Q}^{\ast }\equiv \min_{Q}f\}.
\]%
A powerful algorithm for solving numerically the problem (P) is the well
known Gradient Projection Method ((GP) for a short) which was introduced separately by
Goldestein in 1964 \cite{Gol} and Levitin and Polyak in 1966 \cite{LPo}. This algorithm is defined
recursively as follows:%
\begin{equation}
\left\{
\begin{array}{l}
x_{0}\in Q \\
x_{n+1}=P_{Q}(x_{n}-\gamma _{n}\nabla f(x_{n})),%
\end{array}%
\right.   \tag{GP}
\end{equation}%
where $P_{Q}$ denotes the projection from $\mathcal{H}$\ onto $Q$ and
$(\gamma _{n})$ is a given sequence of positive real numbers. It is known
(see for instance \cite{Xu}) that if $(\gamma _{n})$ satisfies the condition%
\begin{equation}
0<\lim \inf_{n\rightarrow +\infty }\gamma _{n}\leq \lim \sup_{n\rightarrow
+\infty }\gamma _{n}<\frac{2}{L_{f,Q}},  \label{C2}
\end{equation}%
then the sequence $(x_{n})_{n}$ generated by the algorithm (GP) converges
weakly to some element $x_{\infty }$ of $S_{f,Q}.$

Antipin \cite{Ant} studied the continuous version of (GP). He established that any
trajectory $x(t)$ of the the following dynamical system%
\begin{equation}
\left\{
\begin{array}{l}
x^{\prime }(t)+x(t)=P_{Q}(x(t)-\gamma \nabla f((t))), \\
x(0)=x_{0}%
\end{array}%
\right.   \tag{CGP}
\end{equation}%
where $\gamma >0$ and $x_{0}\in Q$, converges weakly as $t$ goes to $+\infty
$ to some minimizer of $f$ over $Q.$ Moreover, he proved that
\[
f(x(t))-f_{Q}^{\ast }\leq \frac{C}{1+t},~\forall t\geq 0,
\]%
for some constant $C>0.$ These results have been improved in \cite{May}
where the case of (CGP) with $\gamma =\gamma (t)$ has been investigated.

In the general case, the convergence of the sequence $(x_{n})$ for the
discreet algorithm (GP) and the trajectory $x(t)$ for the continuous
dynamical system (CGP) are only weak (see \cite{Xu} and \cite{Bai}) and the
corresponding limits are an undefined minimizers of $f$ over $Q$ which may
depend on the initial data $x_{0}.$ To overcome these two weakness many
modifications of the algorithm (GP) and its continuous version (CGP) are
proposed \cite{ACo,Bol,Xu,YXu,CDe,YKJY,YLW}. For instance, in 2002, J. Bolte \cite{Bol} considered the
following dynamical system%
\begin{equation}
\left\{
\begin{array}{l}
x^{\prime }(t)+x(t)=P_{Q}(x(t)-\gamma \nabla f(x(t))-\varepsilon (t)x(t)), \\
x(0) =x_{0},%
\end{array}%
\right.   \tag{CGP$_{\varepsilon }$}
\end{equation}%
where $x_{0}\in Q$ and $\varepsilon :[0,+\infty \lbrack \rightarrow \lbrack 0,+\infty \lbrack
$ is a nonincreasing function converging to zero. He proved hat if $\int_{0}^{+\infty }\varepsilon (t)=+\infty ,\varepsilon
^{\prime }(t)$ is bounded and converges to zero, then every trajectory $x(t)$
of (CGP)$_{\varepsilon }$ converges strongly toward the element of minimal
norm of $S_{f,Q}$. In 2011, H.K. Xu \cite{Xu} studied the asymptotic properties of the
discreet version of (CGP)$_{\varepsilon }$. Precisely, he considered the
following algorithm:
\begin{equation}
x_{n+1}=P_{Q}(x_{n}-\gamma _{n}\nabla f(x_{n})-\gamma _{n}\alpha _{n}x_{n}),
\tag{GP$_{\varepsilon }$}
\end{equation}%
where $(\gamma _{n})_{n}$ and $(\alpha _{n})_{n}$ are nonnegative real
sequences. He established the following strong convergence result.

\begin{theorem}\cite[Hong-Kun Xu]{Xu}
\label{Th0}
Assume that:
\begin{enumerate}
\item[(i)] $0<\gamma _{n}\leq \frac{\alpha _{n}}{(L_{f,Q}+\alpha _{n})^{2}}$
for all $n;$

\item[(ii)] $\alpha _{n}\rightarrow 0$ as $n\rightarrow +\infty ;$

\item[(iii)] $\sum_{n=1}^{+\infty }\alpha _{n}\gamma _{n}=+\infty ;$

\item[(iv)] $\left( \left\vert \gamma _{n}-\gamma _{n-1}\right\vert
+\left\vert \alpha _{n}\gamma _{n}-\alpha _{n-1}\gamma _{n-1}\right\vert
\right) /(\alpha _{n}\gamma _{n})^{2}\rightarrow 0$ as $n\rightarrow +\infty
.$
\end{enumerate}
Then every sequence $(x_{n})_{n}$ generated by the algorithm (GP
)$_{\varepsilon }$ converges strongly to the element of minimal norm of $%
S_{f,Q}.$
\end{theorem}

The main objective of the paper is to improve this theorem by proving a
convergence result for the discreet algorithm (GP)$_{\varepsilon }$ very similar to the
result of Bolte concerning the asymptotic behavior of the trajectories of
the continuous dynamical system (CGP). Indeed, we prove that, if $0<\lim
\sup \gamma _{n}<\frac{2}{L_{f,Q}}$, the sequence $\alpha _{n}$ decreases and converges to zero and  $%
\sum_{n=1}^{+\infty }\gamma _{n}\alpha_{n}=+\infty $, then the sequences $%
(x_{n})_{n}$ generated by the algorithm (GP$_{\varepsilon }$) converge
strongly to the element of minimal norm of $S_{f,Q}.$ Moreover, in the case $%
\sum_{n=1}^{+\infty }\gamma _{n}\alpha_{n}<\infty ,$ we establish a result which
improves the weak convergence criteria (\ref{C2}) for the algorithm (GP).
Precisely, our main result states as follows.

\begin{theorem}
\label{Th}
Let $\phi :\mathcal{H}\rightarrow \mathbb{R}$ be a differentiable
convex function such that it is bounded from below on $Q$ and its gradient function $\nabla\phi$ is $L_{\phi,Q}$-Lipschitz continuous on $Q$. Let $(\gamma _{n})_{n}$ and $(\alpha _{n})_{n}$ be two
sequences of nonnegative real numbers such that $0<\lim \sup \gamma _{n}<%
\frac{2}{L_{f,Q}}$ and $(\alpha _{n})$ decreases and converges to zero. Let $(x_{n})_{n}$ be a sequence
generated by the algorithm%
\begin{equation}
x_{n+1}=P_{Q}(x_{n}-\gamma _{n}\nabla f(x_{n})-\gamma _{n}\alpha _{n}\nabla
\phi (x_{n})).  \tag{GGP$_{\varepsilon }$}
\end{equation}

\begin{enumerate}
\item[(i)] If $\sum_{n=1}^{+\infty }\gamma _{n}=+\infty $ and $%
\sum_{n=1}^{+\infty }\gamma _{n}\alpha _{n}<+\infty ,$ then $(x_{n})$
converges weakly to some element of $S_{f,Q}.$

\item[(ii)] If $\phi$ is strongly convex and $\sum_{n=1}^{+\infty }\gamma _{n}\alpha _{n}=+\infty ,$ then $%
(x_{n})$ converges strongly to the unique minimizer $y^{\ast }$ of $\phi $
over the subset $S_{f,Q}.$
\end{enumerate}
\end{theorem}
\begin{remark}
If we take
\[
\gamma_n=\frac{A}{n^{\gamma}}, \ \alpha_n=\frac{B}{n^{\alpha}},
\]
where $A,B>0$ are absolute constants, then according to the Theorem \ref{Th0} of Xu, the strong convergence of the Algorithm (GP)$_{\varepsilon }$ holds if $0<\alpha<\gamma<1$ and $2\alpha+\gamma<1$; however, Theorem \ref{Th} guaranties the strong convergence of (GP)$_{\varepsilon }$  under the weaker assumptions: $\alpha>0, \gamma\geq0$, and $\alpha+\gamma\leq1$.
\end{remark}
\begin{remark} Theorem \ref{Th} improves \cite[Theorem 3.2]{YLW} where the strong convergence of the algorithm (GP)$_{\varepsilon }$ is established under the following hypothesis: $0<\lim_{n\rightarrow +\infty } \gamma _{n}<%
\frac{2}{L_{f,Q}}$, $\sum_{n=1}^\infty|\gamma_{n+1}-\gamma_n|<\infty$, $\lim_{n\rightarrow +\infty }\alpha_n=0$ and $\sum_{n=1}^\infty\alpha_n=\infty$.
\end{remark}
The rest of the paper is organized as follows. in Section 2, we gather some
general results that will be useful in the proof of our main theorem. The
section 3 is devoted to the proof of the weak convergence property of the
algorithm (GGP)$_{\varepsilon }.$ In the last section, we prove the second
assertion of Theorem concerning the strong convergence property of the
algorithm (GGP)$_{\varepsilon }.$

\section{Some preliminary results}

In this section, we will recall some important results that will be useful
in the proof of the main Theorem \ref{Th}.
Firstly we recall the definition of the operator $P_Q$.
\begin{definition}
For every $x\in \mathcal{H}$, $P_Q(x)$ is the unique element $z\in Q$ which satisfies $\|x-z\|=inf_{y\in Q}\|x-y\|$.
\end{definition}

The following lemma is a well known variational characterization of the
projection operator $P_{Q}$ \cite{CGo}.

\begin{lemma}
\label{Lem1}
Let $x,y\in \mathcal{H}.$ Then $P_{Q}(x)=y$ if and only if $y\in Q$ and $%
\langle y-x,y-v\rangle \leq 0$ for every $v\in Q.$
\end{lemma}
For more details about the properties of the projection operator $P_Q$, we refer the reader to the book of J. Peypouquet \cite{Pey}.
\par\noindent
The second result is a very important weak convergence criteria discovered
independently and almost at the same time by Opial \cite{Opi} and Polyak \cite{Pol}.
\begin{lemma}
[Polyak-Opial's lemma] Let $(x_n)_n$ be a sequence in $\mathcal{H}$. Assume
that there exists a nonempty subset $S$ of $\mathcal{H}$ such that:
\begin{enumerate}
\item[(i)] for every $z\in S,$ $\displaystyle\lim_{n\rightarrow+\infty}\left\Vert
x_n-z\right\Vert $ exists.
\item[(ii)] Every weak cluster point of $(x_n)_n$ belongs to the set $S$,
\end{enumerate}
\noindent Then there exists $x_{\infty}\in S$ such that $(x_n)_n$ converges weakly in $\mathcal{H}$ toward $x_{\infty}$.
\end{lemma}
The third Lemma is a simple criteria for the convergence of nonnegative real sequences.
\begin{lemma}
\label{Lem3}
Let $(x_n)_n$ be a sequence of nonnegative real number. Assume that there exists a non negative real sequence $(\delta_n)_n$ such that $\sum_{n=1}^{+\infty }\delta _n<\infty $ and, for every $n\in \mathbb{N}$,
$x_{n+1}\leq x_n+\delta_n$. Then the sequence $(x_n)_n$ converges.
\end{lemma}
\begin{proof}
It suffices to notice that the sequence $u_n :=x_n+\sum_{k=n}^{+\infty }\delta _k$ is convergent since it is decreasing and bounded from below.
\end{proof}
The following classical result will be used many times in the proof of Theorem \ref{Th}.
\begin{lemma}
\label{Lem4}
Let $g:H\rightarrow R$ be a differentiable convex function such  that
its gradient $\nabla g$ is $L_{g}$-Lipschitz continuous on the convex subset $Q$.  Then for every $x,y\in Q$ we have%
\[
g(y)\leq g(x)+\langle \nabla g(x),y-x\rangle +\frac{L_{g}}{2}\left\Vert
y-x\right\Vert ^{2}.
\]
\end{lemma}

\begin{proof}
Let $x,y\in Q.$ From the fundamental formula of calculus, we have%
\begin{eqnarray*}
g(y) &\leq &g(x)+\int_{0}^{1}\langle \nabla g(x+t(y-x)),y-x\rangle dt \\
&=&g(x)+\langle \nabla g(x),y-x\rangle +\int_{0}^{1}\langle \nabla
g(x+t(y-x))-\nabla g(x),y-x\rangle dt \\
&\leq &g(x)+\langle \nabla g(x),y-x\rangle +\left\Vert y-x\right\Vert
\int_{0}^{1}\left\Vert \nabla g(x+t(y-x))-\nabla g(x)\right\Vert dt.
\end{eqnarray*}%
Since $Q$ is convex, $x+t(y-x)\in Q$ for every $t\in \lbrack 0,1],$ hence
the last inequality implies that
\begin{eqnarray*}
g(y) &\leq &g(x)+\langle \nabla g(x),y-x\rangle +\left\Vert y-x\right\Vert
\int_{0}^{1}L_{g}t\left\Vert y-x\right\Vert dt \\
&= &g(x)+\langle \nabla g(x),y-x\rangle +\frac{L_{g}}{2}\left\Vert
y-x\right\Vert ^{2}.
\end{eqnarray*}
\end{proof}
The last result in this section is a powerful lemma which has been used in many works to prove the strong convergence of variant algorithms related to the fixed point theory of non expansive mappings. A first version of this lemma is firstly given by Bertsekas \cite{Ber}. The following improved version is due to Xu \cite{Xu1}.
\begin{lemma}
\label{Lem5} Let  $(u_{n})_{n},(\varepsilon _{n})_{n},(r_{n})_{n}$ and $(\delta_n)_n$ be three
non negative real sequences such that:
\begin{enumerate}
\item[(1)] $(\varepsilon _{n})_{n}\in \lbrack 0,1]$ and $\sum_{n=0}^{+\infty
}\varepsilon _{n}=+\infty $.
\item[(2)] $r _{n}\rightarrow 0$ as $n\rightarrow +\infty$.
\item[(3)]  $\sum_{n=1}^{+\infty }\delta_n<\infty$.
\item[(4)] For every $n\in\mathbb{N}, u_{n+1}\leq (1-\varepsilon _{n})u_{n}+r_{n}\varepsilon
_{n}+\delta_n$.
\end{enumerate}
Then $u_{n}\rightarrow 0$ as $n\rightarrow +\infty .$
\end{lemma}
\begin{proof}
We give here a new proof of this lemma different from those given in \cite{Ber} and \cite{Xu}. The idea of our proof is inspired by the resolution of the differential inequality of type $%
u^{\prime }(t)\leq -\varepsilon (t)u(t)+r(t)\varepsilon (t).$ Let us first notice that up to replace $u_n$ by $u_n+\sum_{m=n}^{+\infty }\delta_m$ and $r_n$ by $r_n+\sum_{m=n}^{+\infty }\delta_m$, we can assume without loss of generality that $\delta_n=0$ for every $n\in\mathbb{N}$. Now, since $1-\varepsilon _{n}\leq e^{-\varepsilon _{n}},$ we have%
\[
u_{n+1}\leq e^{-\varepsilon _{n}}u_{n}+r_{n}\varepsilon _{n}.
\]%
Then, by induction, we deduce that%
\begin{equation}
u_{n+1}\leq e^{-\Gamma _{n}}u_{0}+e^{-\Gamma _{n}}\sum_{k=0}^{n}e^{\Gamma
_{k}}\varepsilon _{k}r_{k},  \label{IP}
\end{equation}%
where%
\[
\Gamma _{n}=\sum_{k=0}^{n}\varepsilon _{k}.
\]%
Let $0<m<n$ two integers. From (\ref{IP}), we have
\begin{equation}
u_{n+1}\leq e^{-\Gamma _{n}}u_{0}+e^{-\Gamma _{n}}\sum_{k=0}^{m-1}e^{\Gamma
_{k}}\varepsilon _{k}r_{k}+(\sup_{k\geq m}r_{k})e^{-\Gamma
_{n}}\sum_{k=m}^{n}e^{\Gamma _{k}}\varepsilon _{k}  \label{IP2}
\end{equation}%
Let us now notice that for every $k\geq 1,$ we have%
\begin{eqnarray}
\varepsilon _{k}e^{\Gamma _{k}} &=&(\Gamma _{k}-\Gamma _{k-1})e^{\Gamma _{k}}
\nonumber \\
&\leq &e(\Gamma _{k}-\Gamma _{k-1})e^{\Gamma _{k-1}}  \nonumber \\
&\leq &e(e^{\Gamma _{k}}-e^{\Gamma _{k-1}}),  \label{i}
\end{eqnarray}%
where in the last inequality we have used the mean value theorem. Inserting (%
\ref{IP}) into (\ref{IP2}), we obtain%
\[
u_{n+1}\leq e^{-\Gamma _{n}}u_{0}+e^{-\Gamma _{n}}\sum_{k=0}^{m-1}e^{\Gamma
_{k}}\varepsilon _{k}r_{k}+(\sup_{k\geq m}r_{k})e.
\]%
Hence, by letting $n$ then $m$ go to infinity, we get%
\[
\lim \sup_{n\rightarrow +\infty }u_{n}\leq e\lim \sup_{m\rightarrow +\infty
}r_{m},
\]%
which completes the proof of the lemma.
\end{proof}

\section{The weak convergence for the algorithm (GGP)$_{\varepsilon }$ }
In this section, we prove the first assertion of the main theorem concerning
the weak convergence of the algorithm (GGP)$_{\varepsilon }$. The proof
relies essentially on Ployak-Opial's lemma.

\begin{proof}
Set
\begin{equation}
\Phi _{n}(x)=f(x)+\alpha _{n}(\phi (x)-\phi ^{\ast }), \label{KA}
\end{equation}
 where $\phi
^{\ast }=\inf_{Q}\phi .$ Since $x_{n+1}=P_Q(x_n-\gamma_n\nabla\Phi(x_n))$, then according to the variational characterization of
the operator $P_{Q}$
\begin{equation}
\langle x_{n+1}-w,x_{n+1}-x_{n}-\gamma _{n}\nabla \Phi (x_{n})\rangle \leq 0.
\label{H1}
\end{equation}%
for every $w\in Q$. Hence by taking $w=x_{n}$ and using the fact that $\nabla \Phi $ is $L_{n}$-Lipschitz
continuous on $Q$ with $L_{n}=L_{f,Q}+\alpha _{n}L_{\phi ,Q}$, we deduce, thanks to Lemma \ref{Lem4}, that%
\[
(\frac{1}{\gamma _{n}}-\frac{L_{n}}{2})\left\Vert x_{n+1}-x_{n}\right\Vert
^{2}+\Phi _{n}(x_{n+1})-\Phi _{n}(x_{n})\leq 0.
\]%
Since $(\alpha _{n})_{n}$ is decreasing, the last inequality implies%
\begin{equation}
(\frac{1}{\gamma _{n}}-\frac{L_{n}}{2})\left\Vert x_{n+1}-x_{n}\right\Vert
^{2}+\Phi _{n+1}(x_{n+1})-\Phi _{n}(x_{n})\leq 0.  \label{H2}
\end{equation}%
On other hand, since $\lim \sup \gamma _{n}<\frac{2}{L_{f,Q}}$ and $\lim
\alpha _{n}=0,$ there exists $\nu >0$ and an integer $n_{0}$ such
\begin{equation}
\frac{1}{\gamma _{n}}-\frac{L_{n}}{2}\geq \nu \text{ }\forall n\geq n_{0}.
\label{H3}
\end{equation}%
Since we are only concerned with the asymptotic behavior of the sequence $%
(x_{n})_{n},$ we can assume without loss of generality that $n_{0}=1.$
Hence, combining the estimates (\ref{H2}) and (\ref{H3}) yields%
\[
\nu \left\Vert x_{n+1}-x_{n}\right\Vert ^{2}+\Phi _{n+1}(x_{n+1})-\Phi
_{n}(x_{n})\leq 0.
\]%
Therefore the sequence $(\Phi _{n}(x_{n}))_{n}$ is non increasing and
\begin{equation}
\sum_{n=1}^{+\infty }\left\Vert x_{n+1}-x_{n}\right\Vert ^{2}<\infty .
\label{H4}
\end{equation}%
Let $\tilde{x}$ be an arbitrary but fixed element of the set $S_{f,Q}.$
Letting $w=\tilde{x}$ in (\ref{H1}), we get%
\[
\langle x_{n+1}-\tilde{x},x_{n+1}-x_{n}\rangle +\gamma _{n}\langle \nabla
\Phi _{n}(x_{n}),x_{n+1}-x_{n}\rangle +\gamma _{n}\langle \nabla \Phi
_{n}(x_{n}),x_{n}-\tilde{x}\rangle \leq 0.
\]%
Using now the elementary identity%
\[
2\langle a,b\rangle =\left\Vert a\right\Vert ^{2}+\left\Vert b\right\Vert
^{2}-\left\Vert a-b\right\Vert ^{2},
\]%
Lemma \ref{Lem4}, and the fact that $\Phi _{n}$ is a convex function, we easily obtain
\[
\left\Vert x_{n+1}-\tilde{x}\right\Vert ^{2}+2\gamma _{n}\left( \Phi
_{n}(x_{n+1})-\Phi _{n}(\tilde{x})\right) \leq \left\Vert x_{n}-\tilde{x}%
\right\Vert ^{2}+(\gamma _{n}L_{n}-1)\left\Vert x_{n+1}-x_{n}\right\Vert
^{2}.
\]%
This inequality implies%
\begin{equation}
\left\Vert x_{n+1}-\tilde{x}\right\Vert ^{2}+2\gamma _{n}\left( \Phi
_{n+1}(x_{n+1})-f_{Q}^{\ast }\right) \leq \left\Vert x_{n}-\tilde{x}%
\right\Vert ^{2}+\delta _{n},  \label{H5}
\end{equation}%
where
\[
\delta _{n}:=2\gamma _{n}\alpha _{n}(\phi (\tilde{x})-\phi _{Q}^{\ast
})+\gamma _{n}L_{n}\left\Vert x_{n+1}-x_{n}\right\Vert ^{2}.
\]%
From (\ref{H4}) and the assumption $\sum_{n=1}^{+\infty }\gamma _{n}\alpha
_{n}<+\infty ,$ we infer that the series $\sum_{n=1}^{+\infty }\delta _{n}$
is also convergent. Hence, by applying Lemma, we deduce from (\ref{H5}),
that the real sequence $(\left\Vert x_{n+1}-\tilde{x}\right\Vert )_{n}$
converges and
\[
\sum_{n=1}^{+\infty }\gamma _{n}\left( \Phi _{n+1}(x_{n+1})-f_{Q}^{\ast
}\right) <\infty .
\]%
Using now the fact that the sequence $\left( \Phi
_{n+1}(x_{n+1})-f_{Q}^{\ast }\right) _{n}$ is nonnegative and decreasing and
the assumption $\sum_{n=1}^{+\infty }\gamma _{n}=+\infty ,$ we infer that
\[
\lim \Phi _{n+1}(x_{n+1})=f_{Q}^{\ast },
\]%
which implies that
\begin{equation}
\lim f(x_{n})=f_{Q}^{\ast },  \label{H6}
\end{equation}%
thanks to the facts that $(x_{n})$ is bounded and $\alpha _{n}$ converges to
zero. Hence by using the fact that subset $Q$ is weakly
closed and the function $f$ is weakly lower semi-continuity, we deduce that every weak cluster point of the sequence
$(x_{n})$ belongs to the set $S_{f,Q}.$ Therefore Polyak-Opial's
lemma ensures that the sequence $(x_{n})_{n}$ converges weakly towards some element of
$S_{f,Q}.$
\end{proof}
\section{The strong convergence for the algorithm (GGP)$_{\varepsilon }$ }
In this section we prove the main assertion (ii) of Theorem \ref{Th} about the strong convergence of the algorithm (GGP)$_{\varepsilon }$. The main idea of our proof is inspired by \cite{CPS}.
\begin{proof}
Let $\Phi _{n}$ be the function defined by (\ref{KA}) in the previous section. Since
$\Phi _{n}$ is strongly convex, it has a unique minimizer $y_{n}$ over the
set $Q.$ Let $y^{\ast }$ be the unique minimizer of $\phi $ over the closed
and convex subset $S_{f,Q}.$ Let us first prove that $\phi (y_{n})$ converges to $%
\phi (y^{\ast }).$ Since $\Phi _{n}(y_{n})\leq \Phi _{n}(y^{\ast })$ and $%
y^{\ast }$ is a minimizer of $f$ over $Q,$ we have
\[
\phi (y_{n})\leq \phi (y^{\ast }),
\]
which implies that $(y_{n})_{n}$ is bounded in $\mathcal{H}$. Let $%
(y_{n_{k}})_{k}$ be a subsequence of $(y_{n})_{n}$ which converges weakly to
some $\tilde{y}.$ Since $Q$ is weakly closed, $\tilde{y}\in Q.$ On the
other hand, by letting $n_{k}\rightarrow +\infty $ in the inequalities%
\begin{eqnarray}
\Phi _{n_{k}}(y_{n_{k}}) &\leq &\Phi _{n_{k}}(y^{\ast }),  \nonumber \\
\phi (y_{n}) &\leq &\phi (y^{\ast }),  \label{K}
\end{eqnarray}%
and using the weak lower semi-continuity of $f$ and $\phi ,$ we deduce that $%
f(\tilde{y})\leq f(y^{\ast })$ and $\phi (\tilde{y})\leq \phi (y^{\ast }),$
which clearly implies that $\tilde{y}=y^{\ast }.$ Therefore $(y_{n})_{n}$
converges weakly to $y^{\ast }.$ Hence, by using an other time the weak
lower semi continuity of $\phi ,$ we infer that $\phi (y^{\ast })\leq \lim
\inf \phi (y_{n}).$ This inequality combined with (\ref{K}) yields
\begin{equation}
\lim \phi (y_{n})=\phi (y^{\ast }).  \label{K1}
\end{equation}%
Now by proceeding as in the first part of the proof of the assertion (i) of
Theorem \ref{Th}, we obtain
\begin{equation}
\sum_{n=1}^{\infty }\left\Vert x_{n+1}-x_{n}\right\Vert ^{2}<\infty ,
\label{KK}
\end{equation}%
and
\[
\langle x_{n+1}-y^{\ast },x_{n+1}-x_{n}\rangle +\gamma _{n}\langle \nabla
\Phi _{n}(x_{n}),x_{n+1}-x_{n}\rangle +\gamma _{n}\langle \nabla \Phi
_{n}(x_{n}),x_{n}-y^{\ast }\rangle \leq 0.
\]%
Hence by applying Lemma \ref{Lem4} and using the strong convex inequality%
\[
\langle \nabla \Phi _{n}(x_{n}),x_{n}-y^{\ast }\rangle \geq \Phi
_{n}(x_{n})-\Phi _{n}(y^{\ast })+\frac{m\alpha _{n}}{2}\left\Vert
x_{n}-y^{\ast }\right\Vert ^{2},
\]%
where $m>0$ is the strong convexity parameter of the function $\phi ,$ we infer
that
\begin{eqnarray*}
\left\Vert x_{n+1}-y^{\ast }\right\Vert ^{2} &\leq &\left( 1-m\gamma
_{n}\alpha _{n}\right) \left\Vert x_{n}-y^{\ast }\right\Vert ^{2}+\gamma
_{n}L_{n}\left\Vert x_{n+1}-x_{n}\right\Vert ^{2}+2\gamma _{n}\left( \Phi
_{n}(y^{\ast })-\Phi _{n}(x_{n+1})\right)  \\
&\leq &\left( 1-m\gamma _{n}\alpha _{n}\right) \left\Vert x_{n}-y^{\ast
}\right\Vert ^{2}+\gamma _{n}L_{n}\left\Vert x_{n+1}-x_{n}\right\Vert
^{2}+2\gamma _{n}\left( \Phi _{n}(y^{\ast })-\Phi _{n}(y_{n})\right)  \\
&\leq &\left( 1-m\gamma _{n}\alpha _{n}\right) \left\Vert x_{n}-y^{\ast
}\right\Vert ^{2}+\gamma _{n}L_{n}\left\Vert x_{n+1}-x_{n}\right\Vert
^{2}+2\gamma _{n}\alpha _{n}\left( \phi (y^{\ast })-\phi (y_{n})\right) .
\end{eqnarray*}%
Finally by using (\ref{K1}), (\ref{KK}) and the fact that $\left( \gamma
_{n}L_{n}\right) _{n}$ is bounded and applying Lemma \ref{Lem5}, we conclude from the
last inequality that ($x_{n})_{n}$ converges strongly to $y^{\ast }.$ The
proof is complete.
\end{proof}
\par\noindent Conclusion:
\par\noindent In this paper, we have investigated the effect of adding a convex Tikhonov regularizing term $\gamma_n\alpha_n\triangledown\phi$ to the the gradient projection algorithm
\[x_{n+1}=P_{Q}(x_{n}-\gamma _{n}\nabla f(x_{n})).
\]
By following a dynamical approach, we have essentially established that if $\phi$ is strongly convex and the sequence $(\gamma_n\alpha_n)_n$ converges slowly to zero then any  generated sequence $(x_n)_n$ by the modified gradient projection algorithm converges strongly to the unique minimizer of $\phi$ on the set of the minimizers of the objective function $f$ on $Q$.

\end{document}